\documentclass[a4paper,11pt]{amsart}
\usepackage{dsfont,hyperref,fancyhdr,mathrsfs,amsmath,amscd,amsthm,amsfonts,latexsym,amssymb,stmaryrd}
\usepackage[all]{xy}

\newcommand{\ol}{\mathcal{O}}

\def \a{\alpha}

\def \e{\eta}
\def \ep{\varepsilon}

\def \phi{\varphi}
\def \Phi{\varPhi}
\def \p{\pi}
\def \r{\rho}

\def \C{\mathbb{C}\,}

\def\widecheckg{g^{\hspace*{-2.5pt}\vbox to 5pt{\hbox to
0pt{\LARGE$\check{}$}}}\hspace*{2pt}}

\def\widecheckl{\lambda^{\hspace*{-3.5pt}\vbox to 8pt{\hbox to
0pt{\LARGE$\check{}$}}}\hspace*{2pt}}

\begin{document}

\title{Twistor theory for exceptional holonomy}
\author{Radu Pantilie}  
\address{R.~Pantilie, Institutul de Matematic\u a ``Simion~Stoilow'' al Academiei Rom\^ane,
C.P. 1-764, 014700, Bucure\c sti, Rom\^ania} 
\email{\href{mailto:Radu.Pantilie@imar.ro}{Radu.Pantilie@imar.ro}} 
\subjclass[2010]{53C28, 53C29, 53C25} 
\keywords{exceptional holonomy, twistor theory}

\newtheorem{thm}{Theorem}[section]
\newtheorem{lem}[thm]{Lemma}
\newtheorem{cor}[thm]{Corollary}
\newtheorem{prop}[thm]{Proposition}

\theoremstyle{definition}

\newtheorem{defn}[thm]{Definition}
\newtheorem{rem}[thm]{Remark}
\newtheorem{exm}[thm]{Example}

\numberwithin{equation}{section}

\thispagestyle{empty}

\begin{abstract}
We show that the $G_2$-manifolds and certain ${\rm Spin}(7)$-manifolds are endowed with natural Riemannian twistorial structures. 
Along the way, the exceptional ho\-lo\-no\-my representations are reviewed and other related facts are considered.
\end{abstract} 

\maketitle 

\section*{Introduction}  

\indent 
It is well-known that the Weyl curvature tensor is involved in the obstruction to integrability of almost twistorial structures. 
This is due to the fact that, pointwisely, it corresponds to a holomorphic section of a line bundle over the Grassmannian of isotropic 
$2$-dimensional subspaces of the model Euclidean space; this already shows why the dimension $n\geq4$\,. Alternatively, if $n\geq5$\,, 
that line bundle can be described as the square of the very ample generator of the closed adjoint orbit of ${\rm SO}(n,\C\!)$ on 
the projectivisation of its Lie algebra (if $n=4$\,, that Grassmannian is the disjoint union $\C\!P^1\sqcup\C\!P^1$ and the line bundle 
is induced by the line bundle of Chern number $4$ over the Riemann sphere).\\ 
\indent 
In the joint work \cite{DesLouPan}\,, we introduced the notion of `Riemannian twistorial structure' as the necessary 
augmentation that makes the Riemannian manifolds the objects of a category. Furthermore, we have shown that such structures 
can be found on any simply-connected Riemannian symmetric space. In this note, we carry on this programme by showing that, essentially,  
this also holds for any $G_2$-manifold and for any ${\rm Spin}(7)$-manifold whose holonomy group is contained by $G_2$ 
(Theorem \ref{thm:eholon_G2} and Corollary \ref{cor:eholon_G2}\,, below). 
This is, obviously, related to the fact that, in the former case, for example, the Riemannian curvature tensor 
of such a manifold corresponds, at each point, to a holomorphic section 
of a line bundle over the closed adjoint orbit on the projectivisation of the complexified Lie algebra of $G_2$ 
(see \cite{Bry-1987}\,, \cite{Sal-holo_book}\,). Pointwisely, our twistorial structures are given by equivariant holomorphic embeddings  
of the $5$-dimensional nondegenerate complex hyperquadric into the Grassmannians of isotropic subspaces of dimensions $3$ and $4$ of 
the corresponding complexified irreducible representation spaces of $G_2$ and ${\rm Spin}(7)$\,, respectively 
(Proposition \ref{prop:e_Euclid_twist_str}\,). 
Along the way, we discuss other facts related to the exceptional holonomy representations, such as, in 
Section~\ref{section:simple_sym_decomp}\,, a special kind of symmetric decompositions of Lie algebras (compare~\cite{Sal-holo_book}\,).

\section{Simple symmetric decompositions of Lie algebras} \label{section:simple_sym_decomp} 

\indent 
Unless otherwise stated, all the vector spaces (and Lie algebras) are complex; for example, $\mathfrak{sl}(2)=\mathfrak{sl}(2,\C\!)$\,.   
We denote by $U_n$ the irreducible (complex) representation space of $\mathfrak{sl}(2)$ of degree $n\in\mathbb{N}$\,.\\ 
\indent 
A classical way to build Lie algebras goes as follows. Let $\mathfrak{h}$ be a Lie algebra and let $\mathfrak{p}$ be a vector space, 
endowed with a representation $\a$ of $\mathfrak{h}$ on $\mathfrak{p}$\,, and an $\a$-invariant $\mathfrak{h}$-valued $2$-form $\e$ 
on $\mathfrak{p}$\,. Then there exists a unique antisymmetric bilinear form $[\cdot,\cdot]$ on $\mathfrak{h}\oplus\mathfrak{p}$ 
that induces the given Lie bracket on $\mathfrak{h}$ and such that 
$[h,p]=\a(h)(p)$\,, $[p_1,p_2]=\e(p_1,p_2)$\,, for any $h\in\mathfrak{h}$ and $p,p_1,p_2\in\mathfrak{p}$\,.\\ 
\indent 
The \emph{curvature form} of $\mathfrak{h}\oplus\mathfrak{p}$ is the tensor of degree $(1,3)$ on $\mathfrak{p}$ 
given by $R=\a\circ\e$\,. 
Then $[\cdot,\cdot]$ satisfies the Jacobi identity if and only if $R$ satisfies the first Bianchi identity 
$R\wedge{\rm Id}_{\mathfrak{p}}=0$\,. If this holds then the fairly standard terminology is that $\mathfrak{h}\oplus\mathfrak{p}$ 
is a \emph{symmetric decomposition}.\\ 
\indent 
We are interested in a special type of symmetric decompositions which are built as follows (see \cite{Sal-holo_book}\,), 
where $\p_n$ and $\ep_n$ are the $\mathfrak{sl}(2)$-invariant projections from $U_n\otimes U_n$ onto $U_2$ and $U_0$\,, 
respectively, $(n\in\mathbb{N}\setminus\{0\})$\,. 

\begin{prop} \label{prop:simple_sym_decomp} 
Let $\mathfrak{g}=\mathfrak{h}\oplus\mathfrak{p}$ be a symmetric decomposition such that there exist 
$k\in\mathbb{N}\setminus\{0\}$ and $n_1,\ldots,n_k\in\mathbb{N}\setminus\{0\}$\,, with $n_1+\cdots+n_k$ even, satisfying 
the following conditions:\\ 
\indent 
\quad{\rm (i)} $\mathfrak{h}$ is the Lie algebras direct sum of $k$ copies of $\mathfrak{sl}(2)$\,,\\ 
\indent 
\quad{\rm (ii)} $\mathfrak{p}=U_{n_1}\otimes\cdots\otimes U_{n_k}$ and $\a$ is the corresponding irreducible representation 
of $\mathfrak{h}$\,,\\ 
\indent 
\quad{\rm (iii)} $\e=\sum_{j=1}^{j=k}\ep_{n_1}\otimes\cdots\otimes\ep_{n_{j-1}}\otimes\p_{n_j}\otimes\ep_{n_{j+1}}\cdots\otimes\ep_{n_k}$\,.\\  
\indent 
Then $\mathfrak{g}$ is simple unless $\mathfrak{g}=\mathfrak{so}(4)$\,, $k=1$ and $n_1=2$\,.
\end{prop} 
\begin{proof}
Any nonzero element of $\Lambda^2U_1$ induces an ad-invariant Euclidean structure on $\mathfrak{g}$ and, consequently, 
$\mathfrak{g}$ is reductive. Furthermore, $\mathfrak{h}$ reductive, $\a$ faithfull and $\e$ surjective quickly implies that 
the center of $\mathfrak{g}$ is zero, and, thus, $\mathfrak{g}$ is semisimple.\\ 
\indent 
Now, we look at $\mathfrak{g}$ as an $\mathfrak{h}$-representation space, induced by its adjoint representation. 
Similarly, $\mathfrak{p}$ and the ideals of $\mathfrak{g}$ are such representation spaces. As $\mathfrak{p}$ is irreducible, 
it must be contained by just one ideal of $\mathfrak{g}$ and, as $\mathfrak{p}$ generates $\mathfrak{g}$\,, that ideal must be equal 
to $\mathfrak{g}$\,. 
\end{proof}

\indent 
A Lie algebra $\mathfrak{g}$ as in Proposition \ref{prop:simple_sym_decomp} is said to be endowed with a \emph{simple decomposition}. 
Next, we concentrate on the problem of finding which Lie algebras admit simple decompositions.\\ 
\indent  
For this, let $k\in\mathbb{N}\setminus\{0\}$ and $n_1,\ldots,n_k\in\mathbb{N}\setminus\{0\}$\,, 
with $n_1+\cdots+n_k$ even. Then, with $\mathfrak{h}$ and $\mathfrak{p}$ as in (i) and (ii) of Proposition \ref{prop:simple_sym_decomp}\,, 
we denote by the symmetric multi-index $(n_1\ldots n_k)$ the vector space $\mathfrak{h}\oplus\mathfrak{p}$ endowed with the antisymmetric 
bilinear form $[\cdot,\cdot]$\,, given by $\a$ and $\e$ of (ii) and (iii) from Proposition \ref{prop:simple_sym_decomp}\,. 

\begin{thm} \label{thm:simple_sym_decomp} 
The Lie algebras admitting simple decompositions are the following: 
\begin{equation*} 
\begin{split} 
\mathfrak{so}(4)&=(2)\\ 
\mathfrak{so}(5)&=(1.1)\\ 
\mathfrak{sl}(3)&=(4)\\ 
\mathfrak{g}_2&=(3.1)\\ 
\mathfrak{so}(6)&=(2.2)\\ 
\mathfrak{so}(7)&=(2.1.1)\\ 
\mathfrak{so}(8)&=(1.1.1.1)  
\end{split} 
\end{equation*} 
where $\mathfrak{g}_2$ is the simple Lie algebra of dimension $14$\,. 
\end{thm} 
\begin{proof}
From Proposition \ref{prop:simple_sym_decomp} it follows that, excepting $\mathfrak{so}(4)=(2)$\,, the simple decompositions determine 
real Riemannian symmetric spaces which are irreducible, compact and of type I\,. Then the result follows from the corresponding Cartan's
classification.   
\end{proof}

\indent 
From Theorem \ref{prop:simple_sym_decomp}\,, we obtain $\mathfrak{g}_2=(3.1)$ as a Lie subalgebra of $\mathfrak{so}(7)=(2.1.1)$ and 
the latter as a Lie subalgebra of $\mathfrak{so}(8)=(1.1.1.1)$\,. These give the two \emph{exceptional holonomy representations}, 
as we will, shortly, explain. For now, note that, we, also, have the embeddings of Lie algebras $\mathfrak{so}(4)\subseteq\mathfrak{g}_2$\,,  
$\mathfrak{sl}(3)\subseteq\mathfrak{g}_2$ and $\mathfrak{sl}(3)\subseteq\mathfrak{so}(6)$\,.\\ 
\indent 
Further, the canonical representation of $\mathfrak{so}(7)=(2.1.1)$ reads $U_2\oplus(U_1\otimes U_1)$\,. Then together with the 
embedding $\mathfrak{g}_2\subseteq\mathfrak{so}(7)$ this gives the first exceptional holonomy representation. 
As $\mathfrak{g}_2$ is simple and of dimension $14$\,, and by using $\mathfrak{g}_2=(3.1)\subseteq(2.1.1)=\mathfrak{so}(7)$\,, 
it quickly follows that the first exceptional holonomy representation is irreducible (in fact, this is one of the two fundamental 
representations of $\mathfrak{g}_2$\,, and the other is the adjoint representation).\\ 
\indent 
For the second, we write explicitely $(1.1.1.1)=\bigl(\bigoplus_{j=1}^4{}\mathfrak{sl}(2)^j\bigr)\oplus\bigl(\bigotimes_{j=1}^4U_1^j\bigr)$, 
where $\mathfrak{sl}(2)^j$ and $U_1^j$ are copies of $\mathfrak{sl}(2)$ and $U_1$\,, respectively, $j=1,\ldots,4$\,. 
Then the three fundamental representations of $\mathfrak{so}(8)$ involved in the \emph{triality} are: 
$\bigl(U_1^1\otimes U_1^2\bigr)\oplus\bigl(U_1^3\otimes U_1^4\bigr)$\,, 
$\bigl(U_1^1\otimes U_1^3\bigr)\oplus\bigl(U_1^2\otimes U_1^4\bigr)$ and 
$\bigl(U_1^1\otimes U_1^4\bigr)\oplus\bigl(U_1^2\otimes U_1^3\bigr)$\,.  
The fact that these representations are not equivalent under the group of interior automorphisms of $\mathfrak{so}(8)$ follows from 
Lemma \ref{lem:for_triality}\,, below.\\ 
\indent 
Now, let $\mathfrak{so}(7)=(2.1.1)$ be embedded into $\mathfrak{so}(8)=(1.1.1.1)$ through the morphism 
$\bigoplus_{j=1}^3\mathfrak{sl}(2)^j\to\bigoplus_{j=1}^4\mathfrak{sl}(2)^j$\,, 
$(A_1,A_2,A_3)\mapsto(A_1,A_1,A_2,A_3)$\,, for any $A_1,A_2,A_3\in\mathfrak{sl}(2)$\,. Then  the representation of $\mathfrak{so}(7)$ 
induced by this embedding and $\bigl(U_1^1\otimes U_1^3\bigr)\oplus\bigl(U_1^2\otimes U_1^4\bigr)$ is the second exceptional 
holonomy representation (this is the fundamental (irreducible) representation of $\mathfrak{so}(7)$ of dimension $8$\,; the other two are 
the canonical and the adjoint representations, respectively).

\begin{lem} \label{lem:for_triality} 
Let $\mathfrak{so}(6)=(2.2)$ be embedded into $\mathfrak{so}(8)=(1.1.1.1)$ through the morphism 
$\bigoplus_{j=1}^2\mathfrak{sl}(2)^j\to\bigoplus_{j=1}^4\mathfrak{sl}(2)^j$\,, 
$(A_1,A_2)\mapsto(A_1,A_1,A_2,A_2)$\,, for any $A_1,A_2\in\mathfrak{sl}(2)$\,.\\ 
\indent 
Then the representations of $\mathfrak{so}(6)$ induced by this embedding and the representations 
$\bigl(U_1^1\otimes U_1^2\bigr)\oplus\bigl(U_1^3\otimes U_1^4\bigr)$ and 
$\bigl(U_1^1\otimes U_1^3\bigr)\oplus\bigl(U_1^2\otimes U_1^4\bigr)$ are distinct. 
\end{lem}   
\begin{proof}  
We restrict to $\bigoplus_{j=1}^2\mathfrak{sl}(2)^j\subseteq\mathfrak{so}(6)$ the two representations. 
Then the first one gives 
$\bigl(U_2^1\otimes U_0^2\bigr)\oplus\bigl(U_0^1\otimes U_2^2\bigr)\oplus\bigl(U_0^1\otimes U_0^2\bigr)\oplus\bigl(U_0^1\otimes U_0^2\bigr)$\,, 
whilst the second one gives $\bigl(U_1^1\otimes U_1^2\bigr)\oplus\bigl(U_1^1\otimes U_1^2\bigr)$\,. 
\end{proof} 

\begin{rem} \label{rem:for_triality}  
The first representation of Lemma \ref{lem:for_triality} is the direct sum of the canonical representation of $\mathfrak{so}(6)$ 
and the $2$-dimensional trivial representation. Consequently, the second representation restricted to 
$\mathfrak{sp}(4)\subseteq\mathfrak{sl}(4)$ is $U\oplus U$, where $U$ is the canonical 
representation of $\mathfrak{sp}(4)$\,. 
\end{rem}  

\begin{cor} \label{cor:for_octonions}
There exist embeddings of Lie algebras $\mathfrak{so}(5)\subseteq\mathfrak{so}(6)\subseteq\mathfrak{so}(7)$ such that 
$\mathfrak{g}_2\cap\mathfrak{so}(6)=\mathfrak{sl}(3)$ and $\mathfrak{sl}(3)\cap\mathfrak{so}(5)$ is equal to one of the two 
$\mathfrak{sl}(2)$\!'s involved in the simple decomposition of $\mathfrak{so}(5)$\,.  
\end{cor} 
\begin{proof} 
The second embedding and the first equality follow from Theorem \ref{thm:simple_sym_decomp}\,.\\ 
\indent 
To complete the proof we, firstly, make the obtained embedding $\mathfrak{sl}(3)\subseteq\mathfrak{so}(6)$ more explicit. 
For this, note that, $\Lambda^2(U_1\otimes U_1)=U_2\oplus U_2$ gives the well-known identification 
$\mathfrak{sl}(U_1\otimes U_1)=\mathfrak{so}(U_2\oplus U_2)$\,. As $U_1\otimes U_1=U_2\oplus U_0$\,, we obtain the embedding  
$\mathfrak{sl}(U_2)\subseteq\mathfrak{so}(U_2\oplus U_2)$\,.\\ 
\indent 
Now, a suitable basis of $U_1$ makes $U_1\otimes U_1=U_1\oplus U_1$ a symplectic vector space such that 
$\mathfrak{sl}(U_2)\cap\mathfrak{sp}(U_1\oplus U_1)=\mathfrak{sl}(U_1)$\,, and the proof quickly follows. 
\end{proof}

\begin{cor} \label{cor:for_Cayley_cross_product}  
There exists a reductive decomposition $\mathfrak{so}(7)=\mathfrak{g}_2\oplus\mathfrak{p}$ such that the induced 
representation of $\mathfrak{g}_2$ on $\mathfrak{p}$ is the first exceptional holonomy representation. 
\end{cor} 
\begin{proof}
This follows from Corollary \ref{cor:for_octonions} by using the reductive decomposition 
$\mathfrak{so}(5)=\mathfrak{sl}(2)\oplus\mathfrak{p}$ obtained from its simple decomposition. 
\end{proof}

\indent 
Now, a straightforward calculation shows that the `torsion' of the decomposition $\mathfrak{so}(7)=\mathfrak{g}_2\oplus\mathfrak{p}$ 
of Corollary \ref{cor:for_Cayley_cross_product}\,, essentially, is the Cayley cross product \cite{Sal-holo_book}\,. 
This gives a simple proof of the following classical result. 

\begin{thm} \label{thm:Cayley_cross_product} 
The Lie algebra of infinitesimal automorphisms of the Cayley cross product is $\mathfrak{g}_2$\,. 
\end{thm} 
\begin{proof}
We have that $\mathfrak{g}_2$ preserves the Cayley cross product as this is the torsion of the reductive decomposition 
$\mathfrak{so}(7)=\mathfrak{g}_2\oplus\mathfrak{p}$\,. As the representation of $\mathfrak{g}_2$ 
on $\mathfrak{p}$ is irreducible, the proof quickly follows.  
\end{proof}

\section{Twistor theory for exceptional holonomy} 

\indent 
It is quite well-known that all of the Riemannian symmetric spaces infinitesimally described in Theorem \ref{thm:simple_sym_decomp} 
are endowed with quaternionic-like structures \cite{Pan-hqo}\,, and, consequently, with twistorial structures. Less well-known are those 
corresponding to $\mathfrak{sl}(3)$ and $\mathfrak{g}_2$\,. For the former see \cite{Pan-ro-q_tame} (and the references therein) 
whilst the latter is well-known as a (complex) Wolf space. Furthermore, the embedding $\mathfrak{g}_2\subseteq\mathfrak{so}(7)$   
induces on it a Hermitian and a quaternionic-like structure (compare \cite{SveWoo-2015}\,) corresponding to the exact sequence of vector bundles 
$$0\longrightarrow4\ol(-2)\longrightarrow\C\!P^1\times\C^{\!8}\longrightarrow4\ol(2)\longrightarrow0\;,$$ 
where $\ol(-1)$ denotes the tautological line bundle over the Riemann sphere. Note that, the Grassmannian 
${\rm Gr}_3^+(7)$ is the heaven space of $G_2/{\rm SO}(4)$\,, where the former is endowed with its $f$-quaternionic structure 
(see \cite{Pan-hqo}\,, and the references therein) with twistor space the $5$-dimensional nondegenerate hyperquadric $Q^5$, 
and $G_2$ denotes the symply-connected (complex) Lie group with Lie algebra $\mathfrak{g}_2$\,.\\ 
\indent 
The next result, also, involves the Grassmannians ${\rm Gr}_3^0(7)$ and ${\rm Gr}_4^0(8)$ of isotropic subspaces 
of dimensions $3$ and $4$ of the irreducible representations spaces of $\mathfrak{so}(7)$ of dimensions $7$ and $8$\,, respectively, 
and the $6$-dimensional hyperquadric $Q^6$ in the projectivisation of the exceptional holonomy representation space 
of $\mathfrak{so}(7)$\,. 

\begin{prop} \label{prop:e_Euclid_twist_str} 
{\rm (i)} There exists a $\mathfrak{so}(7)$-invariant diffeomorphism from ${\rm Gr}_3^0(7)$ onto~$Q^6$.\\ 
\indent 
{\rm (ii)} There exists a $\mathfrak{g}_2$-invariant embedding of $Q^5$ into ${\rm Gr}_3^0(7)$\,.\\ 
\indent 
{\rm (iii)} There exists an embedding of $Q^5$ into ${\rm Gr}_4^0(8)$\,, which is equivariant with respect to the 
embedding morphism from $\mathfrak{g}_2$ into $\mathfrak{so}(7)$\,.  
\end{prop} 
\begin{proof} 
Assertion (i) follows from Theorem \ref{thm:simple_sym_decomp} by checking that the infinitesimal actions of 
$\mathfrak{so}(7)$ on ${\rm Gr}_3^0(7)$ and $Q^6$ have the same isotropy Lie algebras, with respect to suitably chosen points.\\  
\indent 
Assertion (ii) follows from (i) and Corollary \ref{cor:for_octonions} (or it can be proved directly).\\ 
\indent 
For assertion (iii)\,, note that, the intersection of $\mathfrak{g}_2$ and the isotropy Lie algebra, at a suitably chosen point, 
of the action of $\mathfrak{so}(7)$ on ${\rm Gr}_4^0(8)$ is equal to the isotropy Lie algebra 
of the action of $\mathfrak{g}_2$ on $Q^5$.    
\end{proof} 

\indent 
The embeddings $Q^5\subseteq{\rm Gr}_3^0(7)$ and $Q^5\subseteq{\rm Gr}_4^0(8)$ of Proposition \ref{prop:e_Euclid_twist_str} give 
the Euclidean twistorial structures \cite{DesLouPan} we are interested in.\\ 
\indent 
From now on, as a structural group, $G_2$ is considered either with its exceptional holonomy representation or with the representation 
induced through the embedding $G_2\subseteq{\rm Spin}(7)$\,, where the latter is considered with its exceptional holonomy representation. 

\begin{thm} \label{thm:eholon_G2} 
Any Riemannian manifold with holonomy group contained by $G_2$ is endowed with a Riemannian twistorial structure, 
given at each point by the embeddings $Q^5\to{\rm Gr}_4(7)$\,, $p\mapsto p^{\perp}$, 
for any $p\in Q^5\subseteq{\rm Gr}_3^0(7)$\,, or $Q^5\subseteq{\rm Gr}_4^0(8)$\,, respectively.\\ 
\indent 
Moreover, on the, locally, obtained twistor spaces there exist distributions transversal to the twistor submanifolds 
and of dimensions $3$ or $4$\,, respectively.     
\end{thm}   
\begin{proof}
If $G_2$ is considered with its exceptional holonomy representation, the obstruction to integrability reads 
$R\bigl(p^{\perp},p^{\perp}\bigr)(p^{\perp})\subseteq p^{\perp}$,  
for any $p\in Q^5\subseteq{\rm Gr}_3^0(7)$\,, where $R$ is the Riemannian curvature form (of the Levi-Civita connection).\\ 
\indent 
Also, it is well-known that the irreducible representations of $G_2$ are parametrized by $\mathbb{N}\times\mathbb{N}$\,. 
For example, the first exceptional holonomy representation corresponds to $(1,0)$\,, the adjoint representation to $(0,1)$\,, and 
$R$ is contained by the representation space corresponding to $(0,2)$ (see \cite{Bry-1987}\,, \cite{Sal-holo_book}\,).\\ 
\indent 
On the other hand, the mentioned obstruction to integrability is contained in a direct sum of irreducible representation spaces 
corresponding to pairs from $\bigl(\mathbb{N}\setminus\{0\}\bigr)\times\mathbb{N}$\,, and the proof follows.\\ 
\indent 
The second case and the last statement are similar.   
\end{proof} 

\indent 
For any $G_2$-manifold (where $G_2$ is considered with its exceptional holonomy representation) 
it can be shown that, also, the almost twistorial structure given by $Q^5\subseteq{\rm Gr}_3^0(7)$ is integrable. 
Then the corresponding `CR twistor space' is the twistor space of a ${\rm Spin}(7)$-manifold, with holonomy group 
contained by $G_2$\,, so that the latter manifold appears as the `heaven space' 
of the former (in the real setting, the CR twistor space of the 
$G_2$-manifold is a real hypersurface of the twistor space of the ${\rm Spin}(7)$-manifold; 
compare \cite{Pan-hqo} and the references therein); moreover, there exists 
a natural twistorial retraction of this inclusion map.\\ 
\indent 
The consequences for twistorial harmonic maps can be straightforwardly deduced.\\ 
\indent 
Theorem \ref{thm:eholon_G2} can be strenghten as follows, where, for simplicity, we deal only with the $7$-dimensional case. 

\begin{cor} \label{cor:eholon_G2} 
Let $M$ be a $G_2$-manifold and let $Z$ be its (local) twistor space. Then the $(G_2$-$)$frame bundle of $M$ corrresponds,  
under the Ward transformation, to a principal bundle $\mathcal{P}$ over $Z$. Moreover, there exists a surjective submersion 
$\chi:\mathcal{P}\to Q^5$ whose fibre, over any $\xi\in Q^5$, is the total space of a principal subbundle of $\mathcal{P}$ 
with structural group the isotropy subgroup of $G_2$ at $\xi$.  
\end{cor} 
\begin{proof} 
This follows quickly from Theorem \ref{thm:eholon_G2} and the Ward transformation. 
\end{proof} 

\begin{rem} \label{rem:converse} 
Let $\mathcal{U}$ be the dual of the restriction to $Q^5\subseteq{\rm Gr}_3^0(7)$ of the tautological vector bundle over ${\rm Gr}_3(7)$\,. 
As the restriction of\/ $\mathcal{U}$ to any `associative' conic is isomorphic to $\ol(2)\oplus 2\ol(1)$\,, 
from \cite{Buch} and \cite{GraRem-cas}\,, we deduce that $H^1(\mathcal{U})=0=H^1({\rm End}\,\mathcal{U})$\,.\\ 
\indent 
Let $Z$ be a manifold endowed with an embeding of $Q^5$ with normal bundle (isomorphic to) $\mathcal{U}$. By \cite{Kod}\,, there exists 
a locally complete family of (nondegenerate) $5$-quadrics embedded into $Z$ and containing the given $Q^5\subseteq Z$, as a member. 
Furthermore, by applying \cite{Gri-extension_I} (see \cite{Nar-deformations}\,), we deduce that the normal bundle of each such $5$-quadric is 
$\mathcal{U}$. Denote by $M$ the parameter space of this family.\\ 
\indent 
Suppose that $\mathcal{P}$ is a principal bundle over $Z$ with structural group $G_2$ and endowed with a surjective submersion
$\chi:\mathcal{P}\to Q^5$ whose fibre, over any $\xi\in Q^5$, is the total space of a principal subbundle of $\mathcal{P}$ 
with structural group the isotropy subgroup of $G_2$ at $\xi$.\\ 
\indent 
If the restriction of $\mathcal{P}$ to any twistor $5$-quadric is trivial, we may apply the Ward transformation, thus, obtaining  
a principal bundle $P$ over $M$ with structural group $G_2$\,, and endowed with a connection $\nabla$ compatible with $Z$; that is, 
$\nabla$ restricted to the submanifolds of $M$ corresponding to the points of $Z$ gives flat connections 
(this condition can be, also, formulated, pointwisely, by using the twistorial structure determined by $Z$ on $M$).\\ 
\indent 
Moreover, the existence of $\chi$ implies that $P$ is a reduction of the frame bundle of $M$. Consequently, 
the almost twistorial structure on $M$ given by $Q^5\subseteq{\rm Gr}_4(7)$ is integrable, with respect to $\nabla$.\\ 
\indent 
Furthermore, by using \cite{Bry-G2}\,, we obtain that  
the curvature form $R$ and torsion $T$ of $\nabla$ satisfy, pointwisely, 
\begin{equation} \label{e:conditions} 
\begin{split} 
R&\in U_{0,0}\oplus U_{0,1}\oplus U_{0,2}\\ 
T&\in U_{0,0}\oplus U_{1,0}\oplus U_{0,1}\;, 
\end{split} 
\end{equation} 
where $U_{m,n}$ is the irreducible representation space of $G_2$ corresponding to $(m,n)\in\mathbb{N}^2$.  
(We, further, have that $\nabla$ and the Levi-Civita connection of $M$ are projectively equivalent if and only if 
$T\in U_{0,0}\oplus U_{1,0}$\,.)\\ 
\indent 
Conversely, if a $G_2$-manifold with torsion satisfies \eqref{e:conditions} then its associated almost twistorial structures are integrable 
and it, also, satisfies the conclusion of Corollary \ref{cor:eholon_G2}\,. For the horizontal distribution  on the associated bundle of quadrics 
to be projectable to the twistor space, one further needs the condition $R\in U_{0,2}$ (in which case, the Ricci tensor of $\nabla$ is zero). 
\end{rem} 

\indent 
Similarly to Remark \ref{rem:converse}\,, a partial converse to Theorem \ref{thm:eholon_G2} and Corollary \ref{cor:eholon_G2}\,, the latter, 
suitably, adapted to ${\rm Spin}(7)$-manifolds with holonomy group contained by $G_2\,(\subseteq{\rm Spin}(7)$\,),   
can be obtained.   

\begin{exm} \label{exm:deform} 
Let $Z$ be the projectivisation of the dual of the tautological bundle over $Q^5\subseteq{\rm Gr}_4^0(8)$\,, endowed with the 
images of the sections of the projection $\p:Z\to Q^5$. An open neighbourhood of each such quadric is the twistor space 
of a flat $G_2$-manifold.\\ 
\indent 
By using \cite{San-weak_Fano_unobst}\,, we obtain a family $(Z_t)_{t\in U}$ of compact manifolds, parametrized by an open neighbourhood 
$U$ of $0\in U_{1,0}=H^1(TZ)$\,, and such that $Z_0=Z$. Also, from \cite{Kod} we deduce that we may assume that 
each $Z_t$\,, $(t\in U)$\,, is endowed with a $7$-dimensional family of (nondegenerate) $5$-quadrics; furthermore, 
similarly to Remark \ref{rem:converse}\,, we deduce that the normal bundle of each such quadric is $\mathcal{U}$. 
Moreover, if $t\neq0$ then $Z_t$ is not diffeomorphic to $Z$ (for example, because it has no nontrivial vector fields; 
a consequence of \cite{Gri-extension_I}\,).\\ 
\indent 
For each $\xi\in Q^5$ let $\mathcal{P}^{\xi}$ be the pull-back by $\p$ of $(G_2,Q^5,H^{\xi})$\,, where $H^{\xi}$ is the isotropy subgroup 
of $G_2$ at $\xi$\,. By using \cite{Gri-extension_I}\,, we deduce that we may assume that $\mathcal{P}^{\xi}$  
extends to a principal bundle $\mathcal{P}^{\xi}_t$ over any $Z_t$\,, $t\in U$.\\ 
\indent 
Now, note that, $\mathcal{U}$ is associated to $(G_2,Q^5,H_{\xi})$\,. Thus, to any $\mathcal{P}^{\xi}_t$ we may associate a vector bundle 
$\mathcal{U}_t$ over $Z_t$\,, where, obviously, $\mathcal{U}_0=\p^*\mathcal{U}$. Moreover, by using the fact that 
$H^1({\rm End}\,\mathcal{U})=0$\,, we deduce that the obtained family $(\mathcal{U}_t)_{t\in U}$ is unique. 
Consequently, $\mathcal{P}_t=\mathcal{P}^{\xi}_t\times_{H^{\xi}}G_2$ does not depend of $\xi\in Q^5$.\\ 
\indent 
Thus, $Z_t$ is the twistor space of a $G_2$-manifold with torsion satisfying \eqref{e:conditions}\,, for any $t\in U$.  
In particular, $Z$ is the twistor space of ${\rm Spin}(7)/G_2$ endowed with the canonical connection of the reductive decomposition  
of Corollary \ref{cor:for_Cayley_cross_product}\,. 
\end{exm}


\begin{thebibliography}{10} 

\bibitem{Bry-1987} 
R.~L.~Bryant, Metrics with exceptional holonomy, 
\textit{Ann. of Math. (2)}, {\bf 126} (1987) 525--576. 

\bibitem{Bry-G2} 
R.~L.~Bryant, Some remarks on $G_2$-structures,   
\textit{Proceedings of G\"okova Geometry-Topology Conference 2005}, G\"okova Geometry/Topology Conference (GGT), G\"okova, 2006, 75--109.

\bibitem{Buch} 
N.~Buchdahl, On the relative de Rham sequence, 
\textit{Proc. Amer. Math. Soc.}, {\bf 87} (1983) 363--366.

\bibitem{DesLouPan} 
G. Deschamps, E. Loubeau, R. Pantilie, Harmonic maps and twistorial structures,
\textit{Mathematika}, {\bf 66} (2020) 112--124.  

\bibitem{GraRem-cas} 
H.~Grauert, R.~Remmert, \textit{Coherent analytic sheaves}, 
Grundlehren der Mathematischen Wissenschaften [Fundamental Principles of Mathematical Sciences], 265, 
Springer-Verlag, Berlin, 1984. 

\bibitem{Gri-extension_I} 
P.~A.~Griffiths, The extension problem for compact submanifolds of complex manifolds. I. The case of a trivial normal bundle, 
\textit{Proc. Conf. Complex Analysis (Minneapolis, 1964)}, Springer, Berlin, 1965, 113--142. 

\bibitem{Kod}
K.~Kodaira, A theorem of completeness of characteristic systems for analytic families of
compact submanifolds of complex manifolds,
\textit{Ann. of Math. (2)}, {\bf 75} (1962) 146--162. 

\bibitem{Nar-deformations} 
M.~S.~Narasimhan, Deformations of complex structures and holomorphic vector bundles, 
\textit{Complex analysis, Proc. Summer School, (Trieste, 1980)}, J.~Eells (editor), Lecture Notes in Math., Springer, 1982, 196--209.  

\bibitem{Pan-hqo} 
R.~Pantilie, Quaternionic-like manifolds and homogeneous twistor spaces, 
\textit{Proc. A.}, {\bf 472} (2016) 20160598, 11 pp. 

\bibitem{Pan-ro-q_tame} 
R.~Pantilie, On tame $\r$-quaternionic manifolds, Preprint IMAR, 2019, 
(available from \href{https://arxiv.org/abs/1901.05072}{\tt https://arxiv.org/abs/1901.05072}).  

\bibitem{Sal-holo_book} 
S.~Salamon, \textit{Riemannian geometry and holonomy groups}, 
Pitman Research Notes in Mathematics Series, 201. Longman Scientific \& Technical, Harlow; 
copublished in the United States with John Wiley \& Sons, Inc., New York, 1989, viii+201 pp. 

\bibitem{San-weak_Fano_unobst} 
T.~Sano, Unobstructedness of deformations of weak Fano manifolds, 
\textit{Int. Math. Res. Not. IMRN}, (2014) no. 18, 5124--5133.

\bibitem{SveWoo-2015} 
M.~Svensson, J.~C.~Wood, Harmonic maps into the exceptional symmetric space $G_2/{\rm SO}(4)$\,, 
\textit{J. Lond. Math. Soc. (2)}, {\bf 91} (2015) 291--319.

\end{thebibliography}
\end{document}